\documentclass[11pt, oneside]{article}   	
\usepackage{geometry}                		
\usepackage{graphicx}				
\usepackage{amssymb}
\usepackage{amsmath}
\usepackage{amsthm}
\usepackage{harpoon}
\usepackage{accents}

\usepackage{tikz}  
\usepackage{float}
\restylefloat{table}
\usepackage{mathrsfs}
\usepackage{verbatim}
\usepackage{lmodern}

\usetikzlibrary{positioning,chains,fit,shapes,calc}  

\newtheorem{theorem}{Theorem}
\newtheorem{lemma}{Lemma}
\newtheorem{definition}{Definition}

\newtheorem{remark}{Remark}
\newtheorem{corollary}{Corollary}

\title{On the Bernoulli Numbers via the Newton-Girard Identities}
\author{Mario DeFranco}

\begin{document}
\maketitle 
\abstract{We prove formulas for the Bernoulli numbers by using the Newton-Girard identities to evaluate the Riemann zeta function at positive even integers. To do this, we define a sequence of positive integers, a sequence of polynomials, and a sequence of linear operators on the space of functions. We prove properties of these polynomials, such as the positivity of their coefficients, and present a combinatorial formula for the Bernoulli numbers as a positive sum over plane trees which can be generalized as a transform of sequences. We also combinatorially prove the Newton-Girard identities using the symmetric group.
}

\section{Introduction} \label{Introduction}

The Bernoulli numbers $B_{k}$ for $ k\geq 0$ are a sequence of rational numbers that appears in many areas of mathematics, from topology to number theory. See \cite{Mazur} for an overview of their significance. They are named after Jacob Bernoulli who used them to calculate the power sums
\[
\sum_{n=1}^N n^k
\]
in his book \textit{Ars Conjectandi} published posthumously in  A.D. 1713. See \cite{Bernoulli} for an English translation. Seki Kowa is also credited with independently deriving these numbers (see \cite{Selin}). 

One of the well-known appearances of the Bernoulli numbers is in the evaluation of the Riemann zeta function at positive even integers: 
\begin{equation} \label{zeta bernoulli}
\zeta(2k) = \sum_{n=1}^\infty \frac{1}{n^{2k}} = (-1)^{k-1}\pi^{2k}\frac{2^{2k}B_{2k}}{2(2k)!}.
\end{equation}
The values $\zeta(2k)$ were  first evaluated by L. Euler in A.D. 1740 \cite{Euler}. The proof of equation \eqref{zeta bernoulli} traditionally given in the literature compares two different expansions of the function $\cot(x)$ (see \cite{Dwilewicz}). In this paper, we evaluate $\zeta(2k)$ another way using the Newton-Girard identities. These identities are combinatorial relations between elementary symmetric functions and power-sum symmetric functions. Named after I. Newton and A. Girard, they appear in Newton's \textit{Arithmetica Universalis} \cite{Newton} A.D. 1707 and Girard's paper \cite{Girard} A.D. 1629. Thus by equation \eqref{zeta bernoulli} 
our evaluation for $\zeta(2k)$ also provides formulas for the Bernoulli numbers. We describe this evaluation now.

We define a sequence of positive integers $A_k$ for $k \geq1$: 
\[
1,1,10,945, 992250, 13575766050, 2787683360962500, 9732664704199465153125, \dots 
\]
and prove that 
\[
\zeta(2k) =  \frac{\pi^{2k}}{2} \frac{A_k }{ \prod_{i=1}^k (2i+1)!!}
\]
where 
\[
(2i+1)!! = \prod_{j=1}^i (2j+1).
\]

To obtain $A_k$, we first define polynomials $P_k(x)$ and then define
\[
A_k = P_k(k).
\]
 We list the first five translated polynomials: 
\begin{align*}
P_1(\frac{x}{2}+1-\frac{3}{2}) &= 1\\ 
P_2(\frac{x}{2}+2-\frac{3}{2}) &= 1\\ 
P_3(\frac{x}{2}+3-\frac{3}{2}) &= 7+x \\ 
P_4(\frac{x}{2}+4-\frac{3}{2}) &= 465 + 130 x + 10 x^2\\ 
P_5(\frac{x}{2}+5-\frac{3}{2})&= 360045 + 142695 x + 19845 x^2 + 945 x^3\\ 
& \vdots
\end{align*}
Note that these polynomials have positive coefficients and that the sequence $A_k$ also appears as the leading coefficients. We prove these properties in Section \ref{Pk}.

We define the $P_k(x)$ recursively by defining operators $\mathcal{B}_k$. These constructions naturally arise from the Newton-Girard identities applied to symmetric functions in variables $z_n$ specialized to 
\[
z_n = \frac{1}{n^2}
\]
(see Definition \ref{e p}). In section \ref{B_k}, we present a combinatorial definition of $A_k$ as a sum over plane trees such that each term is positive. Combining this with our combinatorial evaluation \cite{DeFranco2} of 
\[
\zeta (\{ 2\}^k) = \frac{\pi^{2k}}{(2k+1)!}
\]
gives a combinatorial evaluation of $\zeta(2k)$. 
\section{The Newton-Girard Identities}
We first present definitions necessary to prove the Newton-Girard identities and our evaluation of $\zeta(2k)$. 

\begin{definition} \label{e p}
Let $z_1, z_2, ...$ be an infinite sequence of indeterminates.
For integer $k\geq 0$, let $e_k$ denote the elementary symmetric function 
\[
e_k = e_k(z_1,z_2,...) = \sum_{1 \leq n_1<n_2<...<n_k} \prod_{i=1}^k z_{n_i}
\]
with $e_0=1$;
and for $k\geq 1$, let $p_k$ denote the power sum symmetric function 
\[
p_k=p_k (z_1,z_2,...)  = \sum_{n=1}^\infty z_n^k
\]
Let $e_{\mathrm{inc}}(k;j)$ denote the incomplete $k$-th elementary symmetric function 
\[
e_{\mathrm{inc}}(k;j) = e_k(z_1, z_2, ... , z_{j-1}, z_{j+1}, ...).
\] 
Let $S_k$ denote the symmetric group on the set $\{1,2,..,k \}$. For $\sigma \in S_k$, let $p_\sigma$ denote
\[
p_\sigma = \prod_{C \in \sigma} p_{|C|}
\]
where $C$ denotes a cycle of $\sigma$ containing $|C|$ elements. For $|C|=n$, we say that $C$ has length $n$, or that $C$ is an $n$-cycle. We also let $\mathrm{sgn}(\sigma)$ denote the signature of the permutation
\[
\mathrm{sgn}(\sigma) = \prod_{C \in \sigma}(-1)^{|C|-1}.
\]
We let $\overline{e}_k$ and $\overline{p}_k$ denote the specializations of these functions at 
\[
z_n = \frac{1}{n^2}.
\] 
\end{definition}

\begin{definition} 
Define the linear operator $d_2$ by
\[
d_2(z_n) = z_n^2
\]
and extend $d_2$ to act on monomials as a derivation. 
\end{definition}

The next theorem is a well-known evaluation of the elementary symmetric function in terms of the power-sum symmetric functions. Our proof below is similar to that presented in \cite{DeFranco} applied to derivatives of the Gamma function, and to the one of K. Boklan \cite{Boklan}. 
\begin{theorem} \label{cycle index}
\[
e_k(z) =  \frac{1}{k!}\sum_{\sigma \in S_k} \mathrm{sgn}(\sigma)p_\sigma
\]
\end{theorem}

\begin{proof}
We use induction on $k$. The statement is true for $k=1$. Assume it is true for some $k \geq 1$. Then we obtain $e_{k+1}$ from $e_k$ by first multiplying $e_k$ by $p_1$: 
\[
p_1 e_k = (k+1)e_{k+1}+ \sum_{j=1}^\infty z_j^2 e_{\mathrm{inc}}(k-1;j).
\] 
Now, since
\[
d_2(e_k)=\sum_{j=1}^\infty z_j^2 e_{\mathrm{inc}}(k-1;j),
\]
we obtain
\begin{equation} \label{k+1 e}
(p_1-d_2)e_k = (k+1)e_{k+1}.
\end{equation}

Now we compute $\displaystyle (p_1-d_2)e_k $ another way. The action of $d_2$ on $p_n$ is 
\[
d_2(p_n) = np_{n+1}.
\]
We claim
\begin{equation} \label{Sk Sk+1}
(p_1-d_2)\sum_{\sigma \in S_k} \mathrm{sgn}(\sigma)p_\sigma= \sum_{\sigma \in S_{k+1}} \mathrm{sgn}(\sigma)p_\sigma.
\end{equation}
Let $\sigma \in S_k$.
Multiplying by $p_1$ corresponds to adjoining the 1-cycle consisting of the element $k+1$ to $\sigma$. The action of $d_2$ on $p_\sigma$ corresponds to creating new permutations by adjoining $k+1$ to each cycle $C$ of $\sigma$; if $C$ is of length $n$, then there are $n$ ways to do this. Increasing the length of one cycle of $\sigma$ by 1 creates a new permutation with signature opposite to that of $\sigma$. This proves the claim.

Using the induction hypothesis, equation \eqref{Sk Sk+1} implies 
\begin{equation} \label{Sk+1}
(p_1-d_2)e_k = \frac{1}{k!}\sum_{\sigma \in S_{k+1}} \mathrm{sgn}(\sigma)p_\sigma.
\end{equation}
Combining equations \eqref{k+1 e} and \eqref{Sk+1} completes the induction step and proof.
\end{proof}

We next prove the Newton-Girard identities by partitioning the symmetric group. 
\begin{theorem}
\[
 (-1)^{k-1}p_k=ke_k -\sum_{i=1}^{k-1} (-1)^{i-1}e_{k-i} p_i 
\]
\end{theorem}
\begin{proof}
 We prove 
 \[
 ke_k = \sum_{i=1}^{k} (-1)^{i-1}e_{k-i} p_i. 
 \]
 From Theorem \ref{cycle index}, this is equivalent to 
 \[
\sum_{\sigma \in S_k} \mathrm{sgn}(\sigma)p_\sigma = \sum_{i=1}^k (-1)^{i-1}(i-1)!{k-1 \choose i-1}p_i\sum_{\sigma \in S_{k-i}} \mathrm{sgn}(\sigma)p_\sigma
 \]
 On the right side, we interpret a term of the form 
 \[
 p_i p_\sigma
 \] 
 for $\sigma \in S_{k-i}$ as corresponding to a permutation $\sigma' \in S_k$ such that the element $k$ is in an $i$-cycle $C$ of $\sigma'$, and $\sigma'=\sigma$ when restricted to the   elements not in $C$. There are $ \displaystyle {k-1 \choose i-1}$
 ways to choose the elements that are in the cycle $C$ and $(i-1)!$ ways to construct the cycle. And 
 \[
 \mathrm{sgn}(\sigma') =  (-1)^{i-1}\mathrm{sgn}(\sigma).
 \]
 This completes the proof.
\end{proof}

\section{The polynomials $P_k(x)$} \label{Pk}
\subsection{Evaluating $\zeta(2k)$}
We have the well-known evaluation of $\overline{e}_k$: 
\[
\overline{e}_k = \frac{\pi^{2k}}{(2k+1)!}.
\]
See \cite{DeFranco2} for a combinatorial proof of this evaluation. Since 
\[
\overline{p}_1 = \overline{e}_1,
\]
we can thus use the Newton-Girard identities to successively solve for $\overline{p}_n$ in terms of the $\overline{p}_i$ for $i <n$ and the $\overline{e}_k$.
We consider the partial sums in the Newton-Girard identities and prove a formula for them in Theorem \ref{Fnk}. We define terms for that theorem next, including the recursive definition of the polynomials $P_k(x)$. 
\begin{definition}
For integer $n\geq 2$ and $k \geq 1$, define $F_n(k)$ by
\begin{align*}
F_n(k) &= k\overline{e}_k -\sum_{i=1}^{n-1} (-1)^{i-1}\overline{e}_{k-i} \overline{p}_i \\ 
            &=  \frac{k\pi^{2k}}{(2k+1)!}  -\sum_{i=1}^{n-1} (-1)^{i-1}\frac{\pi^{2k-2i}\zeta(2i)}{(2k-2i+1)!}.
\end{align*}
\end{definition}


\begin{definition}
Define
\[
P_1(x) = 1
\]
and for $k \geq 1$
\begin{equation} \label{Pn+1}
P_{k+1}(x) = \frac{P_k(k)(\prod_{i=1}^{k} (2x-2k+2i+1))-(\prod_{i=1}^{k} (2i+1))P_k(x)   }{2x-2k}.
\end{equation}
\end{definition}
Note that $P_{k+1}(x)$ is a polynomial because the numerator of equation \eqref{Pn+1} vanishes at $x=k$. 

\begin{theorem} \label{Fnk}
For integer $n\geq 2$ and $k \geq 1$,
\[
F_n(k) = (-1)^{n-1} \frac{\pi^{2k}}{2}P_n(k) \frac{\prod_{i=1}^n (2k-2i+2)}{(2k+1)!\prod_{i=1}^{n-1} (2i+1)!!}.
\] 
\end{theorem}
\begin{proof}
We use induction on $n$.
We have from the evaluation of $\overline{e}_k$ that
\[
\overline{e}_1 = \overline{p}_1 = \zeta(2) = \frac{\pi^2}{3!}.
\]
For $n=2$ we have 
\[
F_2(k)=\frac{k\pi^{2k}}{(2k+1)!} - \frac{\pi^{2k}}{(2k-1)!3!} =-\pi^{2k}\frac{2k(2k-2)}{3!(2k+1)!}. 
\]
Since $P_2(k)=1$, this proves the statement for $n=2$. Assume the statement is true for some $n \geq 2$. Then this 
implies by the Newton-Girard identities that
\[
\overline{p}_n = (-1)^{n-1}F_n(n).
\] 
Thus
\begin{align*}
F_{n+1}(k) &= F_n(k) -(-1)^{n-1}\overline{e}_{k-n}\overline{p}_n\\
&= F_n(k) -\pi^{2k-2n}\frac{F_n(n)}{(2k-2n+1)!}.
\end{align*}
Using the induction hypothesis, this becomes 
\begin{align*}
 &(-1)^{n-1} \frac{\pi^{2k}}{2\prod_{i=1}^{n-1}(2i+1)!!}\\ 
 &\times \left( \frac{(2n+1)!P_n(k)\prod_{i=1}^n (2k-2i+2) - P_n(n) (\prod_{i=1}^n 2i)\prod_{i=1}^{2n} (2k+1-i)}{(2k+1)!(2n+1)!}\right).
\end{align*}
The quantity in parentheses simplifies to 
\begin{align*}
&\frac{(\prod_{i=1}^n (2i)(2k-2i+2))(2k-2n)}{(2n+1)!(2k+1)!} \left( \frac{P_n(k)\prod_{i=1}^n (2i+1) - P_n(n)\prod_{i=1}^n (2k-2n+2i+1)  }{2k-2n}\right)\\
=& \frac{\prod_{i=1}^n (2k-2i+2)}{(2n+1)!!(2k+1)!} (-P_{n+1}(k)).
\end{align*}
Putting this together proves the induction step. This completes the proof. 
\end{proof}

\begin{corollary} For integer $k \geq 1$,
\[
\zeta(2k) =  \frac{\pi^{2k}}{2} \frac{P_k(k)}{\prod_{i=1}^{k}(2i+1)!!}
\]
\end{corollary}
\begin{proof}
We have by the Newton-Girard identities for $k\geq 2$ 
\[
\overline{p}_k = (-1)^{k-1}F_k(k). 
\]
We then evaluate $F_k(k)$ using the theorem. We check that the statement is also true for $k=1$. This completes the proof.
\end{proof}
\subsection{The operators $\mathcal{B}_k$} \label{B_k}

The recursive definition of $P_k(x)$ motivates the following definition of the operator $\mathcal{B}_k$. 
\begin{definition}
 For integer $k \geq 1$ and a function $f(x)$, define the operator $\mathcal{B}_k$ by
 \[
\mathcal{B}_k(f)(x) = \frac{f(k)(\prod_{i=1}^k (2x-2k+2i+1)) - f(x)\prod_{i=1}^k (2i+1)}{2x-2k}.
 \] 
\end{definition}
We thus can define the $P_k(x)$ by 
\[
 P_{k+1}(x) = \mathcal{B}_k \mathcal{B}_{k-1} ... \mathcal{B}_1 (1).
 \]
\begin{lemma} \label{u ai}
Let $u$ and $a_i$ for $1 \leq i \leq k$ be indeterminates. Then
\[
\prod_{i=1}^k (u+a_i)-\prod_{i=1}^k a_i = u\sum_{j=1}^{k} ((\prod_{i=1}^{j-1} (u+a_i) )\prod_{i=j+1}^k a_i )
\] 
where we interpret an empty product to be equal to 1.
\end{lemma}
\begin{proof}
We use induction on $k$. The statement is true for $k=1$. Assume it is true for some $k \geq 1$. Then we have 
\begin{align*}
\prod_{i=1}^{k+1} (u+a_i) &= (u+a_{k+1})\prod_{i=1}^{k} (u+a_i) \\ 
&=u\prod_{i=1}^{k} (u+a_i)+ a_{k+1}\left(u\sum_{j=1}^{k} ((\prod_{i=1}^{j-1} (u+a_i) )\prod_{i=j+1}^{k} a_i )+ \prod_{i=1}^{k} a_i \right)\\ 
&=u\prod_{i=1}^{k} (u+a_i)+ \left(u\sum_{j=1}^{k} ((\prod_{i=1}^{j-1} (u+a_i) )\prod_{i=j+1}^{k+1} a_i )+ \prod_{i=1}^{k+1} a_i \right)\\ 
&= \left(u\sum_{j=1}^{k+1} ((\prod_{i=1}^{j-1} (u+a_i) )\prod_{i=j+1}^{k+1} a_i ) \right)+ \prod_{i=1}^{k+1} a_i.
\end{align*}
This proves the induction step and completes the proof. 
\end{proof}

Next we define terms necessary to state Lemma \ref{B action}.
\begin{definition} 

For an integer $k \geq 1$, let $R(k)$ denote the set 
\[
R(k) = \{3,5,7, ... , 2k+1 \}
\]
with $R(0) = \emptyset$.
For a set $S$ of integers and an integer $m$, let $S+m$ denote the set 
\[
\bigcup_{s\in S} \{s+m\}
\]
where $S+m= \emptyset$ if $S=\emptyset$.
Given $k\geq 2$, suppose $S$ is a set of integers such that 
\[
S \subset R(k-2).
\]
Let $j$ be an integer $0 \leq j \leq k - 1-|S|$. Let $S_{\mathrm{low}}(j,k)$ denote the set consisting of the numbers in $S+2$ and the $j$-th smallest numbers in 
 \[
 R(k-1) - (S+2)
 \]
 with $S_{\mathrm{low}}(0,k) = S+2$. Let $$S_{\mathrm{high}}(j,k)$$ denote the set consisting of the numbers in $S+2$ and the $j$-th highest numbers in 
 \[
 (R(k-1)+2) - (S+2)
 \]
with $S_{\mathrm{high}}(0,k)= S+2$.
Define for non-empty $S$  
\[
\Pi S = \prod_{s \in S} s
\]
and for $S = \emptyset$
\[
\Pi S =1.
\]

Let $f_{S,k}(x)$ denote
\[
 f_{S;k}(x)= \prod_{s \in S} (2x-2k+ s).  
 \]
\end{definition}

\begin{lemma}  \label{B action}
For $k \geq 2$, suppose $S \subset R(k-2)$. 
 Then 
 \[
 \mathcal{B}_k (f_{S;k-1})(x) = \sum_{j=0}^{k -1-|S|} (\Pi S_{\mathrm{high}}(k-1-|S|-j,k) )f_{S_{\mathrm{low}}(j,k); k}(x)  
 \]
\end{lemma}
\begin{proof}
 Applying the definition of $f_{S;k}(x)$ we have 
 \[
 f_{S;k-1}(x) = f_{S+2;k}(x).
 \]
 Then
 \begin{align*}
 \mathcal{B}_k (f_{S+2;k})(x) &= \frac{f_{S+2;k}(k)(\prod_{i=1}^k (2x-2k+2i+1)) - f_{S+2;k}(x)\prod_{i=1}^k (2i+1)}{2x-2k}\\ 
 &= \frac{(\prod_{s \in S+2} s(2x-2k+s)) \Big((\prod_{s \in R(k)-(S+2)} (2x-2k+s)) - \prod_{s \in R(k)-(S+2)} s\Big)}{2x-2k}.
 \end{align*}
 Now we apply Lemma \ref{u ai} with 
 \[
 u = 2x-2k
 \]
 and $a_i$ the $i$-th smallest number in the set 
 \[
 R(k) - (S+2)
 \] 
 for $1 \leq i \leq k-|S|$. This completes the proof.
\end{proof}

\begin{theorem}
For integer $k \geq 2$, the polynomial $P_{k}(x)$ is a positive linear combination of functions of the form $f_{S;k-1}(x)$ where $S \subset R(k-2)$.  
\end{theorem}
\begin{proof}
We use induction on $k$. The statement is true for $k=2$ as
\[
P_2(x) = 1 = f_{\emptyset,1}(x).
\] 
The induction step follows from Lemma \ref{B action}.
\end{proof}

\begin{corollary} \label{positive coefficients}
The polynomial $P_k(x+k-\frac{3}{2})$ has positive coefficients in $x$.
\end{corollary}
\begin{proof}
For $S \subset R(k-2)$, the function $f_{S;k-1}(x)$ is either $1$ or a product of factors of the form 
\[
(2x-2k+m)
\]
where $m \geq 3$. By the theorem, $P_k(x)$ is a positive linear combination of functions $f_{S;k-1}(x)$. This completes the proof.
\end{proof}

We use Lemma  \ref{B action} to express $P_k(x)$ as a sum of positive terms over the set $\mathcal{T}_k$ of plane trees with $k$ vertices. To each tree $T$ we associate two finite sets of integers, $\mathrm{Low}(T)$ and $\mathrm{High}(T)$. For the trees $T$ consisting of one or two vertices, we set 
\[
\mathrm{Low}(T) = \mathrm{High}(T)= \emptyset.
\] 
Suppose $T\in \mathcal{T}_k$ for $k \geq 3$ and let $v$ be the last vertex of $T$ traversed in the preorder.  Say that $v$ is at the $i$-th level of $T$, where $i$ is the number of edges on the path between $v$ and the root. So $1 \leq i \leq k-1$. Let $T'$ denote 
\[
T' = T \backslash v.
\]
 Then set $\mathrm{Low}(T) $ to be the set consisting of the elements in $\mathrm{Low}(T') +2$ and the $k-i-1$ smallest elements in $R(k-2) - (L(T')+2)$; and set $\mathrm{High}(T)$ to be the set consisting of the elements in $\mathrm{Low}(T')+2$ and the $i-1$ greatest elements in $(R(k-2)+2) - (\mathrm{Low}(T')+2)$.
Now define the weight of $T$ to be 
\[
\mathrm{wt}(T) = \mathrm{wt}(T')  \Pi( \mathrm{High}(T))
\]
with $\mathrm{wt}(T) = 1$ for $T \in \mathcal{T}_1$ or $\mathcal{T}_2$.
Then 
\begin{equation} \label{P tree}
P_k(x) = \sum_{T \in \mathcal{T}_k} \mathrm{wt}(T) f_{\mathrm{Low}(T);k-1}(x)
\end{equation}
and thus
\begin{equation} \label{A tree}
 A_k =P_k(k)= \sum_{T \in \mathcal{T}_k} \mathrm{wt}(T)\Pi(\mathrm{Low}(T)+2).
\end{equation}

\begin{theorem} \label{leading}
For integer $k \geq 2$, the leading coefficient of $P_k(x)$ is
\[
A_{k-1}2^{k-2}.
\] 
\end{theorem}
\begin{proof} 
For $k \geq 2$, $P_k(x)$ has degree $k-2$. In the sum \eqref{P tree}, the only trees that contribute a term of $x^{k-2}$ are the those trees whose last vertex $v$ in the preorder is at level 1. For such trees $T$ 
\[
\mathrm{Low}(T) = R(k-2) \text{ and } \mathrm{High}(T) =  \mathrm{Low}(T')+2.
\]
The leading coefficient of $P_k(x)$ is thus 
\begin{align*}
&2^{k-2} \sum_{T \in \mathcal{T}_k, \mathrm{level}(v)=1}  \mathrm{wt}(T)\\ 
&=2^{k-2}  \sum_{T' \in \mathcal{T}_{k-1}} \mathrm{wt}(T') \Pi(\mathrm{Low}(T')+2)\\ 
&=2^{k-2} A_{k-1} 
\end{align*}
by formula \eqref{A tree}. This completes the proof.
\end{proof}

\begin{remark} \emph{
We can express the rational sequence $\displaystyle2\frac{\zeta(2k)}{\pi^{2k}} $ as a transform of the sequence $R = \{R_n\}_{n=1}^\infty$ given by 
\[
R_n =2n+1, \,\,\,\, n\geq 1.
\]  
We write
\begin{equation}\label{zeta rational}
2\frac{\zeta(2k)}{\pi^{2k}} = \frac{\sum_{T \in \mathcal{T}_k} \mathrm{wt}_R(T)}{\prod_{j=1}^k \Pi R(j)}
\end{equation}
where we define $R(k)$ as above, but for $\mathrm{wt}_R(T)$ we interpret the sets $\mathrm{Low}(T)$ and $\mathrm{High}(T)$ as subsets of $R$; for such a subset $S$ we write
\[
S = \{ R_{i_1},..., R_{i_n}\}.
\]
We may then express the operation $S+2$ as
\[
S+2 = \{ R_{i_1+1},..., R_{i_n+1}\}.
\]
The sequence \eqref{zeta rational} can thus be generalized by varying the sequence $R$. 
}

\end{remark}

\subsection{A recursive relation}

Next we prove a linear recursive relation among the coefficients of $P_k(x)$ in the basis $f_{R(n);k-1}(x)$. We prove the following lemma necessary for the recursion. 

\begin{lemma} \label{2ni}
\[
\prod_{i=1}^n (u+ 2i+3) = \sum_{i=0}^n 2^{n-i} \frac{n!}{i!} \prod_{j=1}^i (u+2j+1)
\] 
\end{lemma}
\begin{proof}
 Evaluating at $u=-3$, we get that both sides are equal to $n!2^n$. Evaluating at $u = -2m-3$ for $1\leq m \leq n$, we get that the left side is 0 and that the right side is 
 \[
n! 2^n\sum_{i=0}^{m}(-1)^i {m \choose i} =0.
 \] 
 Both sides are polynomials in $u$ of degree $n$ that are equal at $n+1$ values of $u$. Therefore both sides are equal as polynomials. This completes the proof.
\end{proof}




\begin{theorem}
For integer $k \geq 2$, let
\[
P_k(x) = \sum_{i=0}^{k-2} c_{i,k}\prod_{j=1}^{i} (2x-2k+2j+1).
\] 
with 
\[
c_{0,2} = 1.
\]
Then the coefficients $c_{i,k}$ satisfy 
\[
c_{i,k+1} =(\prod_{j=i+1}^{k-1}(2j+3))  \sum_{n=0}^{i}   (\prod_{j=1}^{n}(2j+1) )(\sum_{m=n}^{k-2}c_{m,k}2^{m-n} \frac{m!}{n!}) 
\]
\end{theorem}
\begin{proof}
We have
\[
\prod_{j=1}^{i} (2x-2k+2j+1)  = f_{R(i); k}(x).
\]
Then the theorem follows directly from Lemmas \ref{B action} and \ref{2ni}.  
\end{proof}
\section{Further Work}

\begin{itemize}

\item Use these formulas or others (such as the Euler zig-zag numbers) to show that 
\[
\sum_{i=0}^n (-1)^i {n \choose i} \frac{\zeta(2k+2i)}{\pi^{i}}
\]
is positive. These expressions arise from the constants 
\[
\sum_{n=1}^\infty \frac{e^{-\pi n^2}}{(\pi n^2)^k} 
\]
after expressing the exponential using the derangement numbers. These constants arise from expansions of the Riemann xi function. 

\item Find eigenvectors of the operators $\mathcal{B}_k$.

\item Vary the sequence $R$ and see if the transforms have asymptotics or generating functions analogous to those of the Bernoulli numbers. 

\item Recover the recurrence relation and generating function for the Bernoulli numbers from these formulas.

\item See if the proofs for the Newton-Girard identities using the symmetric group can be generalized to other Weyl groups. 

\end{itemize}

\end{document}